\documentclass[12pt,reqno]{amsart}
\usepackage{fullpage}
\usepackage{amssymb}
\usepackage{stmaryrd}
\usepackage{fullpage}
\usepackage{amscd}
\usepackage[all]{xy}
\usepackage{comment}
\usepackage{mathrsfs}
\usepackage{verbatim}
\usepackage{longtable}
\usepackage{enumitem}
\usepackage{tikz}
\usepackage[pagewise]{lineno}
\usepackage{lscape}
\usepackage[linesnumbered,ruled,vlined]{algorithm2e}
\usepackage{colonequals}

\newcommand{\ZZhat}{\widehat{\mathbb{Z}}}

\usetikzlibrary{positioning, calc}

\usepackage{color}
\newcommand{\taylor}[1]{{\color{blue} \sf $\spadesuit\spadesuit\spadesuit$ Taylor: [#1]}}

\usepackage[backref,pdfauthor={Taylor Dupuy}]{hyperref}
\usepackage[ruled,vlined]{algorithm2e}

\newcommand{\C}{{\mathbb C}}
\newcommand{\F}{{\mathbb F}}

\newcommand{\PP}{{\mathbb P}}
\newcommand{\Q}{{\mathbb Q}}

\newcommand{\Z}{{\mathbb Z}}
\newcommand{\Qbar}{{\overline{\Q}}}

\newcommand{\Kbar}{{\overline{K}}}

\newcommand{\Fbar}{{\overline{F}}}

\newcommand\td\widetilde

\newcommand{\FF}{{\mathcal F}}

\newcommand{\OO}{{\mathcal O}}

\def\presuper#1#2%
{\mathop{}%
\mathopen{\vphantom{#2}}^{#1}%
\kern-\scriptspace%
#2}

\def\Q{\mathbb{Q}}
\def\C{\mathbb{C}}

\def\P{\mathbb{P}}

\def\Z{\mathbb{Z}}

\def\<{\ensuremath{\langle}}
\def\>{\ensuremath{\rangle}}

 \DeclareMathOperator{\Aut}{Aut}

\DeclareMathOperator{\ord}{ord} 
\DeclareMathOperator{\Div}{Div}

\DeclareMathOperator{\Spec}{Spec}

\newcommand{\GL}{\operatorname{GL}}
\newcommand{\SL}{\operatorname{SL}}

\newcommand{\PSL}{\operatorname{PSL}}

\newcommand\djunion\amalg

\numberwithin{equation}{section}
\newtheorem{theorem}{Theorem}[subsection]
\newtheorem{lemma}[theorem]{Lemma}

\newtheorem{mochizuki}[theorem]{Mochizuki's Inequality}
\theoremstyle{definition}
\newtheorem{definition}[theorem]{Definition}

\newtheorem{example}[theorem]{Example}

\theoremstyle{remark}
\newtheorem{remark}[theorem]{Remark}

\newcommand{\QQ}{\mathbb{Q}}
\newcommand{\ZZ}{\mathbb{Z}}
\newcommand{\RR}{\mathbb{R}}
\renewcommand{\FF}{\mathbb{F}}

\newcommand{\QQbar}{\overline{\mathbb{Q}}}

\newcommand{\hull}{\operatorname{hull}}

\newcommand{\vu}{\underline{v}}
\newcommand{\Vu}{\underline{V}}

\renewcommand{\L}{\mathbb{L}}

\newcommand{\lgp}{\operatorname{lgp}}

\newcommand{\deghatu}{\underline{\widehat{\deg}}}

\newcommand{\bad}{\operatorname{bad}}

\newcommand{\Cu}{\underline{C}}
\newcommand{\Eu}{\underline{E}}
\newcommand{\Xu}{\underline{X}}

\newcommand{\Mu}{\underline{M}}
\newcommand{\epsu}{\underline{\epsilon}}

\makeatletter
\newcommand*{\da@rightarrow}{\mathchar"0\hexnumber@\symAMSa 4B }
\newcommand*{\da@leftarrow}{\mathchar"0\hexnumber@\symAMSa 4C }
\newcommand*{\xdashrightarrow}[2][]{%
	\mathrel{%
		\mathpalette{\da@xarrow{#1}{#2}{}\da@rightarrow{\,}{}}{}%
	}%
}
\newcommand{\xdashleftarrow}[2][]{%
	\mathrel{%
		\mathpalette{\da@xarrow{#1}{#2}\da@leftarrow{}{}{\,}}{}%
	}%
}
\newcommand*{\da@xarrow}[7]{%
	\sbox0{$\ifx#7\scriptstyle\scriptscriptstyle\else\scriptstyle\fi#5#1#6\m@th$}%
	\sbox2{$\ifx#7\scriptstyle\scriptscriptstyle\else\scriptstyle\fi#5#2#6\m@th$}%
	\sbox4{$#7\dabar@\m@th$}%
	\dimen@=\wd0 %
	\ifdim\wd2 >\dimen@
	\dimen@=\wd2 %
	\fi
	\count@=2 %
	\def\da@bars{\dabar@\dabar@}%
	\@whiledim\count@\wd4<\dimen@\do{%
		\advance\count@\@ne
		\expandafter\def\expandafter\da@bars\expandafter{%
			\da@bars
			\dabar@ 
		}%
	}%
	\mathrel{#3}%
	\mathrel{%
		\mathop{\da@bars}\limits
		\ifx\\#1\\%
		\else
		_{\copy0}%
		\fi
		\ifx\\#2\\%
		\else
		^{\copy2}%
		\fi
	}%
	\mathrel{#4}%
}
\makeatother

\newcommand\subsetsim{\mathrel{%
			\ooalign{\raise0.2ex\hbox{$\subset$}\cr\hidewidth\raise-0.8ex\hbox{\scalebox{0.9}{$\sim$}}\hidewidth\cr}}}

\newcommand{\Lbar}{\overline{L}}
\newcommand{\chr}{\operatorname{char}}

\newcommand{\lognubar}{\overline{\ln\nu}}

\renewcommand{\mod}{\operatorname{mod}}

\renewcommand{\FF}{\mathbb{F}}

\newcommand{\Comm}{\operatorname{Comm}}
\newcommand{\can}{\operatorname{can}}

\renewcommand{\int}{\operatorname{int}}

\newcommand{\elink}[1]{\href{https://www.lmfdb.org/EllipticCurve/Q/#1}{\textsf{#1}}}

\SetKwInput{KwInput}{Input}                
\SetKwInput{KwOutput}{Output}              

\begin{document}

\title{The Statement of Mochizuki's Corollary 3.12: Initial Theta Data 
}
\author{Taylor Dupuy and Anton Hilado}


\date{\today}
\thanks{}

\begin{abstract}
We show that Mochizuki's initial theta data is computable from an elliptic curve defined over $\QQ$. 
We work out the case of initial theta data for the elliptic curve with Cremona label \textsf{11a1} in detail.
\end{abstract}

\setcounter{tocdepth}{1}
\maketitle
\tableofcontents

\section{Introduction}\label{S:introduction}

Surrounding Mochizuki's work on ABC conjecture is a notion of \emph{initial theta data}. 
Initial theta data consists of an elliptic curve with some auxiliary data that forms the hypothesis of Corollary 3.12 of \cite{IUT3} --- Mochizuki's famous and controversial inequality (see statement~\ref{C:cor3p12} below).

The present paper is about explaining what Mochizuki's initial theta data is and how to construct examples of it starting with an elliptic curve over $\QQ$. 
This is a necessary prerequisite to understanding the statement of Mochizuki's Corollary 3.12. 

Part of the difficulty of parsing this definition is its length and large number of auxillary constructions. 
Our contribution is to break this definition down into digestible pieces and provide concrete examples.

We begin by stating a version Mochizuki's inequality.

\begin{mochizuki}[Mochizuki's Inequality {\cite[Corollary 3.12]{IUT3}}]\label{C:cor3p12}
	For an elliptic curve $E$ over a number field $F$ sitting in initial theta data $$(\Fbar/F, E_F, l, \Mu, \Vu, V^{\bad}_{\mod}, \epsu)$$ 
	the following inequality holds
	\begin{equation}\label{E:cor3p12}
	-\deghatu(P_q) \leq -\deghatu_{\lgp}(P_{\hull{(U_{\Theta})}}).
	\end{equation}
\end{mochizuki}
The left hand side of \eqref{E:cor3p12} is the degree $\deghatu$ of an Arakelov divisor $P_q \in \widehat{\Div}(F)$ normalized by $[F:\QQ]$ associated to the minimal discriminant of an elliptic curve $\Delta^{\min}_{E/F}$ and is, up to a multiple of $2\ell$, for some auxillary prime number $\ell$, equal to the left hand side of the additive statement of Szpiro's inequality, $\log \vert \Delta^{\min}_{E/F}\vert $. The right hand side is a Mochizuki-log-volume $\lognubar_{\L}(\hull(U_{\Theta}))$ corresponding to the hull of the multiradial representation of the theta-pilot object. 
It is major construction of Mochizuki's theory which is described in \cite[Theorem 3.11]{IUT3}.
We will not address this complicated construction in this paper.

Equation \eqref{E:cor3p12} is a compact version of what is stated in the PRIMS volume. 
Mochizuki's Corollary 3.12 actually spans two pages of the PRIMS volume \cite{IUT3} (pages 597--598)  and last sentence of it corresponds to \eqref{E:cor3p12}.
In the Corollary 3.12 environment of \cite{IUT3} the hypotheses of the Theorem 3.11 environment of \cite{IUT3} are assumed. 
The theorem environment for Theorem 3.11 runs seven pages of the PRIMS volume (page 573 to page 580).
This then in turn depends on the official definition environment given for initial theta data which in the PRIMS volume is Definition 3.1 of Mochizuki's first paper \cite{IUT1} which is the subject of the present paper. 
As stated previously, it is complicated. It spans three pages starting on 68 and ending on page 71. This key definition that is a major obstruction for experts (see for example the discussion  \href{https://mathoverflow.net/questions/299845/explicit-example-of-elliptic-curve-of-the-kind-needed-for-iutt}{MathOverflow 299845}). 

We will now say a few words about how one gets from Mochizuki’s inequality
for elliptic curves in general position to the full ABC conjecture.
This happens in a series of Lemmas embedded in proof environment of \cite[Corollary 2.2]{IUT4}.\footnote{We count approximately nine statements which could be broken into sublemmas. The statement environment of which runs from pg 671 to pg 674 and the proof runs from pg 674 to pg 679.} 
The claim is that Mochizuki's Corollary 3.12, which holds for elliptic curves in initial theta data, implies Vojta's conjecture on a compactly bounded subset $\mathcal{K} \subset U(\overline{\QQ})$ where $U=\PP^1\setminus \lbrace 0,1,\infty\rbrace$.
The terminology ``compactly bounded subset'' was introduced in Mochizuki's \emph{Arithmetic Elliptic Curves in General Position} paper \cite{Mochizuki2010}. 
This specific case of Vojta's inequality for curves implies the general case of Vojta's inequality for curves. 
We will not review \cite{Mochizuki2010} here and refer to the excellent exposition \cite{Matthes2013} but just state that \cite{Mochizuki2010} relies on the result of Mochizuki's \emph{Noncritical Beyli Maps} paper \cite{Mochizuki2004}.

The fact that Vojta's conjecture implies the full ABC conjecture can be found in Bombieri-Gubler \cite[Theorem 14.4.16]{Bombieri2006}. See also Frankenhuysen \cite{VanFrankenhuysen2002}.

The basic idea of \cite[Corollary 2.2]{IUT4}, is to realize $U$ as a parameter space for elliptic curves of the form $y^2 = x(x-1)(x-\lambda)$ then convert the Arakelov degrees in \eqref{E:cor3p12} into heights and the Szpiro-like upper bounds given in \cite[Theorem 1.10]{IUT4} to the log-conductor and log-different terms that appear in Vojta's inequality.
The main claim then is that for each $\mathcal{K}$ of a certain type there exists a finite subset $S \subset \mathcal{K}$ such that for every $\lambda \in \mathcal{K}\setminus S$ of bounded degree $d$, every $\mathcal{E}_{\lambda}$ given by $y^2 = x(x-1)(x-\lambda)$ can be put into initial theta data in such a way that the estimates of \cite[Corollary 3.12]{IUT3} apply with uniform parameters.
This placement into initial theta data then implies Vojta's conjecture a compactly bounded subset of $U$ (and hence ABC).

This present paper explains what initial theta data actually is and gives examples. 
In \S\ref{S:initial-theta-data} we break down this definition into pieces and show that constructing initial theta data starting from an elliptic curve over $\QQ$ is computable.
In particular we prove the following.
\begin{theorem}
	\begin{enumerate}
	\item There exists an algorithm for generating initial theta data ``built from the Field of moduli'' from databases of elliptic curves over $\QQ$. (pseudo-code in \S\ref{S:algorithm})
	\item The tuple $(\Fbar/F, E_F, l, \Mu, \Vu, V^{\bad}_{\mod}, \epsu)$ described in \S\ref{S:e11a1} gives initial theta data for the elliptic curve with Cremona label \textsf{E11a1}.
	\end{enumerate}
\end{theorem}
Besides breaking down the two page definition and removing auxillary constructions, a simplification is the replacement of $\Cu_K$ with $\Mu$ which allows us to relating initial theta data to the theory of the classical modular curve $X_{0}(l)$ .
This viewpoint eliminates the use of the algebraic stack $\underline{C}_K$ and makes the connection to Mochizuki's original idea of a ``global multiplicative subspace'' more transparent. 
We have been extremely careful with the  bookkeeping due to historical confusion and hope that it provides a stable reference for this material. 
In particular we give many definitions and prove an existence theorem for the ``fake global multiplicative subspace'' which uses the classical theory of transvections. 
Later in the verification of a ``non-arithmeticity'' condition we make essential use of a Theorem of Sijsling which characterizes $j$-invariants of elliptic curves we need to exclude \cite{Sijsling2019}. 

\section{Global Multiplicative Subspaces, Initial Theta Data, and E11a1}\label{S:initial-theta-data}
In this section we explain the meaning of the ``Initial Theta Data'' parameters
 $$(F, l, E, S, \Vu, \Mu, \epsu),$$
appearing in Mochizuki's theory.
At a 0th order approximation, these are just constants satisfying enough hypotheses so that a zoo of anabelian interpretations apply simultaneously. 
Our exposition breaks the tuple $(F, l, E, S, \Vu, \Mu, \epsu)$ into two parts: $(F, l, E, S)$ and $(\Vu, \Mu, \epsu)$.

The first chunk, $(F, l, E, S)$, which we call ``pre theta data''  consists of an elliptic curve $E$ over a field $F$ together with a rational prime $l$ and a collection of places of bad multiplicative reduction $S$. 
These parameters are then asked to satisfy some conditions related to the $l$-torsion that allow Mochizuki's anabelian machine to run. 
These should be thought of ``non-isotriviality conditions over the field with one element''.

The second part of the tuple, $(\Vu,\Mu,\epsu)$, serves as a simulation for a ``global multiplicative subspace'' and ``global canonical generator'' (the precise definitions are given in \S\ref{S:transvections}).
For Mochizuki, these serve as a fix to a failure of his approach to the ABC conjecture in his Hodge-Arakelov paper \cite{Mochizuki1999}. 
There, his inequalities suffered from so-called ``Gaussian poles'' rendering them useless.

Finally, after the definition of initial theta data is given, we construct an explicit tuple $(F, l, E, S, \Vu, \Mu, \epsu)$ where $E$ is a base change of the elliptic curve over $\Q$ with Cremona label E11a1.

\subsection{Notation For Galois Actions on Torsion Points of Abelian Varieties}
Throughout this paper if $L$ is a field then $\overline{L} \supset L$ will denote an algebraic closure and we will let  $G_L = G(\overline{L}/L)$ will denote its absolute Galois group. 

For an abelian variety $A$ and an integer $m$ we let $A[m]$ denote its group-scheme of $m$-torsion points. 
Suppose now that $A$ is defined over a field $L$. 
We may abusively write $A[m]$ for the $G_L$-module $A[m](\overline{L})$. 
We will let $\rho_m: G_L \to \Aut(A[m])$ denote the natural $G_L$-representation.
For $m$ coprime to the characteristic of $L$ we refer to $\rho_m$ as the \emph{mod $m$ Galois representation} and we have $A[m]\cong (\Z/m)^{2g}$ as abelian groups.

\subsection{Fricke Involutions}\label{S:fricke-covers}
In the process of unpacking Mochizuki's ``initial theta data'' we make use of the classical theory of the Fricke involution, which we now recall.
Let $Y_0(l)$ be the moduli stack\footnote{
	We note that pairs $(E,M)$ consisting of an elliptic curve together with an $l$-torsion subgroup have an involution which implies the moduli problem is not fine. 
	One can define a course space as a quotient of $Y_1(N)$ which is fine. 
	See for example \cite[\S7.1]{Parson2003}.
} whose $S$-points are elliptic curves with cyclic subgroup of order $l$; by ``cyclic subgroup of order $l$" we mean a finite flat subgroup scheme over $S$ that \'etale locally is isomorphic the constant group scheme $\ZZ/l\ZZ$.
If $S=\Spec(L)$ is the spectrum of a number field then the cyclic subgroup $M \subset E$ in the moduli problem is equivalent to the cyclic subgroup $M(\Lbar) \subset E[l](\Lbar)$ of order $l$ which is stable under $G(\Lbar/L)$. 
\begin{definition}
	The \emph{Fricke involution} is the endomorphism $Y_0(l) \to Y_0(l)$ is given by 
	$$(E,M) \mapsto (\Eu,\Mu):=(E/M, E[l]/M). $$
\end{definition}
We note that $E$ and $\Eu$ are connected by an isogeny
 $$f:E \to \Eu:=E/M$$
given by modding out by the cyclic subgroup $M \subset E$. 
We will call this isogeny the \emph{Fricke isogeny} or \emph{Fricke cover}.
Note that $\Eu/\Mu = (E/M)/(E[l]/M) = E/E[l]=E$ as the cokernel of multiplication by $l$ is $E$; 
one can also see that $\Eu[l]/\Mu\cong M$ under these identifications. 
This show that the Fricke involution is indeed an involution.

Finally, we remark that in the stable compactification $X_0(l)$ of $Y_0(l)$ there are exactly two cusps. 
The first of these is a nodal cubic and the second of these is a chain of $l$ copies of $\P^1$'s in a loop.
In the extension of the Fricke involution to the compactification, these two curves in involution.

\subsection{Arithmetic Fuchsian Groups and Arithmetic Riemann Surfaces}
For two subgroups $\Gamma_1$ and $\Gamma_2$ of $\PSL_2(\RR)$ we say $\Gamma_1$ and $\Gamma_2$ are \emph{isogenous} (=\emph{commensurable}) and write $\Gamma_1 \sim \Gamma_2$ if and only if $[\Gamma_1:\Gamma_1\cap \Gamma_2]<\infty$ and $[\Gamma_2: \Gamma_1 \cap \Gamma_2]< \infty$
. 
Also for a subgroup $\Gamma_0 \subset \PSL_2(\RR)$ we write $\Comm(\Gamma_0) = \lbrace \gamma \in \PSL_2(\RR): \gamma \Gamma_0 \gamma^{-1} \sim \Gamma_0 \rbrace$ for the \emph{commensurator}.

\begin{definition}
	A  Fuchsian group $\Gamma \subset \PSL_2(\RR)$ is called \emph{arithmetic} if and only if 
	$[\Comm(\Gamma):\Gamma]=\infty$. 
\end{definition}

\begin{remark} In \S 2 of \cite{Mochizuki1998} Mochizuki proved that two other notions are equivalent to being arithmetic. Let $\Gamma \subset \PSL_2(\RR)$  be Fuchsian of the first kind (this means the associated fundamental domain has finite volume).
 The following are equivalent.
\begin{enumerate}
	\item $\Gamma$ is arithmetic.
	\item $\Gamma$ is \emph{Shimura arithmetic}: there exists some  $\OO \subset B$ an order in an indefinite quaternion algebra over a totally real field (which we regard as embedded in $M_2(\RR)$) such that $\overline{\OO \cap \SL_2(\RR)} \sim \Gamma.$
	Here the overline denotes the image of this group in $\PSL_2(\RR)$.
	\item $\Gamma$ is \emph{Margulis arithmetic}: \cite[Definition 2.2]{Mochizuki1998}
\end{enumerate}
\end{remark}

We now give some examples of arithmetic Fuchsian groups. Before proceeding we recall briefly that a group $\Gamma$ is \emph{congruence} if and only if it contains $\Gamma(n)$ for some $n$ where $\Gamma(n)$ is the kernel of the reduction map $\PSL_2(\ZZ) \to \PSL_2(\ZZ/n)$.
A group is \emph{non-congruence} if and only if it is finite index in $\PSL_2(\ZZ)$ and not congruence.
\begin{example}\label{E:E11a1}
	All congruence and non-congruence subgroups are arithmetic. 
	The congruence groups include $\Gamma_1(n)$ and $\Gamma_0(n)$.
	Arithmeticity follows from the fact that they are finite index inside $\PSL_2(\ZZ)$ and the behavior of intersections of finite index subgroups.
\end{example}

Let $Y$ be a pointed hyperbolic curve over $\C$ (viewed as a Riemann surface).
Let $\widetilde{Y}=\mathcal{H}$ be its universal cover and let $\Gamma_Y$ be the image of $\pi_1(Y) \to \Aut(\mathcal{H}) = \PSL_2(\RR)$. 

\begin{definition}
	We say that $Y$ is an \emph{arithmetic Riemann surface} if and only if $\Gamma_Y$ is an arithmetic Fuchsian group.
\end{definition}

In what follows, for a Fuchsian group $\Gamma$ we will let $\Gamma \backslash \mathcal{H}$ denote the Riemann surface quotient and we will let $[\Gamma \backslash \mathcal{H}]$ denote the analytic stack (complex orbifold). 
We recall that the coarse space of $[ \Gamma \backslash \mathcal{H}]$ is $\Gamma \backslash \mathcal{H}$. 
For further references about complex analytic stacks and uniformization we refer the reader to \cite[Chapter 6]{VZB}.
\begin{example}
	Let $E$ be the elliptic curve with Cremona label \elink{11a1}, viewed as a Riemann surface.
	Let $X = E \setminus o$ where $o$ is the origin. 
	We have that $E = X_0(11)(\C):= \Gamma_0(11)\backslash\mathcal{H^*}$ where $\mathcal{H}^*$ is the extended upper half plane. 
	One is tempted now to think that $X$ is arithmetic, but this is not the case.
	We give two reasons. 
	First, $Y_0(11)(\C)= \Gamma_0(11)\backslash \mathcal{H}$ is missing two points at infinity. 
	This means one can not compactify by a single point.
	Second, $\Gamma_0(11)$ has elliptic points. 
	We are tempted to say that $\pi_1(X) \cong \Gamma_0(11)$ but this is not the case. We have $\pi_1(X) \cong \mathbb{Z}*\mathbb{Z} $ is the free (non-abelian) group on two generators (since $X$ is homotopic to the wedge of two cirlces). 
	This implies that $\Gamma_X$ must be torsion free and we know that $\Gamma_0(11)$ has elliptic points. 
	We \emph{do} however have that $\pi_1([\Gamma_0(11)\backslash\mathcal{H}]) \cong \Gamma_0(11).$
\end{example}

Arithmetic Riemann surfaces are rare. 
A result of \cite{Takeuchi1983} cited at the end of \S2 of \cite{Mochizuki1998} states that for each $(g,r)$ there are only finitely many hyperbolic arithmetic Riemann surfaces of type $(g,r)$. 
In fact, building on the work of Takeuchi, Sijsling in \cite{Sijsling2019} gives a complete table of punctured elliptic curves whose analytifications are arithmetic Riemann surfaces (and also gives models of these curves over $\QQ$!).

\begin{theorem}[{\cite[Table 4]{Sijsling2019}} Sijsling's Invariants] \label{T:sijsling}
	Let $E$ be an elliptic curve viewed as a complex Riemann surface. Let $o\in E$ be its origin. 
	The Riemann surface $X=E-o$ is arithmetic if and only if 
	 $$j_E  \in \lbrace 0, \quad 1728,\quad 2^{14}31^{3}/5^{3},\quad
	 2^{2}73^3/3^4 \rbrace $$
\end{theorem}

Interestingly, all of these curves have models over $\QQ$.
The $\QQ$ models of the compatified Riemann surfaces  that are not CM (i.e. when $j\neq 0,1728$) are the two curves \elink{20.a1}, \elink{24.a3}. 
In \cite{Sijsling2019} he shows the punctured variants of \elink{20.a1}, \elink{24.a3} are both  arithmetic Riemann surfaces of non-congruence type.

\subsection{Initial Theta Data (Long Form Part 1)}
For $L/\Q_p$ be a finite extension with uniformizer $\pi$ we will be using \emph{normalized valuations} on $K$ we mean the valuation $\ord_L$ satisfying $\ord_L(\pi)=1$. 

Initial Theta Data will consist of a large tuple 
 $$ (F, l, E, S, \Vu, \Mu, \epsu).$$
We will break this tuple into two parts.
First, fix the following.
\begin{itemize}
	\item Let $F$ be a number field with algebraic closure $\Fbar$, and $l$ a rational prime.
	\item Let $E$ be an elliptic curve over $F$.\footnote{Mochizuki uses $X_F = X = E\setminus o$ the punctured elliptic curve.
	Since elliptic curves have identities this data is implicit in $E$.}
	\item Let $K = F(E[l])$.\footnote{Which is also $\cong\Fbar^{\ker(\rho_l)}$.}
	\item Fix $(E,M) \in Y_0(l)(F)$ and $\lbrace\epsu,-\epsu\rbrace \in \Eu[l](\Kbar)/\pm1$.
\end{itemize}

We now give conditions of the first part of the theta data.
\begin{definition}
	We say that $(F,l,E,S)$ is \emph{pre theta data} if and only if 
	\begin{enumerate}
		\item (Non-arithmetic) \label{I:shimura} $E_{\C}\setminus o$ is not an arithmetic Riemann surface. Here $o$ is the identity of the elliptic curve.
		\item \label{I:galois} (Non-isotrivial) $F/\Q(j_E)$ is Galois and $\rho_l(G_F) \subset \Aut(E[l](\Fbar))$ contains a subgroup isomorphic to $\SL_2(\F_l)$. 
		\item (Torsion Conditions) \label{I:cond-2} $\sqrt{-1}\in F$ and $E(F) \supset E(\Fbar)[30]$
		\item (Places of Multiplicative Reduction) $S \subset V(\Q(j_E))_0$ is a non-empty set of places such that every $v \in S$ has odd residue characteristic and every model $E_{\Q(j_E)}$ of $E$ over $\Q(j_E)$ has multiplicative reduction at $v$.  
		\item (Congruence Conditions) 
		Let $E_{\Q(j_E)}$ be a model over the field of moduli. 
		Let $S \subset V(\Q(j_E))$ be an admissible collection of bad places as above. 
		For $l$ the prime above, and all $v \in S$ we require
		\begin{enumerate}
			\item $l \nmid \chr(v)$ for all $v \in S$,
			\item $l \nmid \ord_v(q_v)$,
			\item $l\geq 5$,
			\item $l \nmid [F:\Q(j_E)]$.
		\end{enumerate}
	\end{enumerate}
\end{definition}
We give a couple remarks to aid the reader.
\begin{remark}
	\begin{enumerate}
		\item If we follow \cite[Definition 3.1]{IUT1} directly, condition \ref{I:cond-2} should use 6-torsion rather than 30-torsion. 
		Since \cite{IUT4} later imposes a condition on 30-torsion, we have decided to impose this condition presently.
		\item Condition \ref{I:galois} precludes elliptic curves having CM geometrically (as Galois representations of such curves have abelian image).
		\item This data is free of conditions on $\Vu$. 
		This means at this point we are free to choose $\Vu$ to be whatever.  
	\end{enumerate}
\end{remark}

In order to formulate the conditions on the remaining part of the tuple $(\Vu,\Mu,\epsu)$ we recall some facts about torsion in the Tate uniformization. 

\subsection{The Tate Uniformization}
For $L/\Q_p$ be a finite extension with uniformizer $\pi$ we continue using normalized valuations on $L$ so that $\ord_L(\pi)=1$. 

Recall that for all elliptic curves $A$ over a field $L$, a finite extension of $\Q_p$, with $\vert j_A \vert_L >1$ we have an isomorphism $A(\overline{L}) \cong E_q(\overline{L})$ where $E_q$ is the Tate curve for some unique $q\in \overline{L}$ called the Tate parameter (a geometric invariant which can be written as a formal power series in $1/j_E$ with integer coefficients).
These parameters are important to us because of their relation to the minimal discriminant: $\ord_L(q) = \ord_L(\Delta_{E_q})$ and each $E_q$ is a minimal Weierstrass model. 
All of this is in \cite[Ch V]{Silverman2013}. See also \cite[\S3.5]{Dupuy2020b}.
Later in \S5 will we make use of the rigid analytic uniformizations $\phi: \overline{L}^{\times}/q^{\Z} \to E_q(\overline{L})$ to study torsion points.

\subsection{Notation For Places and Completions }
Throughout this paper if $L$ is a number field we will let $V(L)$ denote the collection of places of $L$. 
We let $V(L)_0$ denote the finite places and $V(L)_{\infty}$ denote the infinite places. 
If $f^*:L_0 \to L$ is an inclusion of fields there is a natural map $f:V(L) \to V(L_0)$. 
If $v_0 \in V(L_0)$ and $f^*$ as above we will use the notation $V(L)_{v_0} = \lbrace v\in V(L): f(v)=v_0\rbrace$.

For $w \in V(L)$ we let $L_w$ denote the completion of $L$ with respect to this place. 
We will let $G_w = G_{L_w}$ and often make a choices of embeddings $\overline{L} \subset \overline{L_w}$ so that $G_{w} \subset G_L$.

\subsection{Multiplicative Subspaces and Canonical Generators}\label{S:transvections}

The following ``canonical splitting'' of the torsion representation for Tate curves is needed both to make definitions and for explicit computations of initial theta data in \S\ref{S:e11a1}.

The point here is we can find a unipotent matrix in the mod $\ell$ representation of the Tate curve. That is there exists some $P_1,P_2 \in E[\ell]$ and $\sigma\in G_L$ such that 
$$\rho_l(\sigma) = \left[ \begin{matrix} 1 & 1 \\ 
0 & 1 \end{matrix} \right], $$
with respect to the basis 
$$P_1 \leftrightarrow \left [ \begin{matrix} 1 \\ 0 \end{matrix} \right] , P_2 \leftrightarrow \left [ \begin{matrix} 0 \\ 1 \end{matrix} \right].$$

\begin{definition}
	An element $\sigma$ together with a choice of basis $P_1$ and $P_2$ is called a \emph{transvection}.
\end{definition}

\begin{lemma}[{\cite[Chapter V]{Silverman2013}}]
	Let $L/\Q_p$ be a $p$-adic field. 
	Let $\ord_L$ be its normalized valuation.
	Let $E/L$ be an elliptic curve with $\ord_L(j_E)<0$ (so that it does not have potentially good reduction).
	Let $l\geq 3$ be a prime not dividing $\ord_L(j_E)$.
	There exists some $\sigma \in I_{\Lbar/L}$ and generators $P_1,P_2 \in E[l](\Lbar)$ such that 
	 $$ \sigma(P_1) = P_1, $$
	 $$ \sigma(P_2) = P_1 + P_2. $$
\end{lemma}

\begin{proof}
The basic strategy is to 
\begin{enumerate} \item show we can reduce to the case where we have split multiplicative reduction and $\zeta_l \in L$, and 
\item prove the split case. 
\end{enumerate}
The proof of the split case goes as follows: Let $q$ be the Tate parameter of $E/L$ and let $Q = q^{1/l} \in \Lbar$ be an $l$th root. 
Then $L(Q,\zeta_l)/L(\zeta_l)$ is a Kummer extension and $G(L(Q,\zeta_l)/L(\zeta_l)) = \langle \sigma \rangle$ with 
 $$ \sigma(Q) = \zeta_l Q.$$
We will make use of the fact that the Tate uniformization $\phi:\Lbar^{\times}/q^{\Z} \to E(\Lbar)$ is an isomorphism of $G_L$-modules; using this isomorphism we have explicit description of torsion via $E(\Lbar)[l] \xleftarrow{\phi} (\zeta_l^{\Z}Q^{\Z})/q^{\Z}$.
We set 
$$P_1 = \phi(\zeta_l), P_2 = \phi(Q),$$
and compute:
 $$\sigma(P_1) = \sigma(\phi(\zeta_l)) = \phi(\sigma(\zeta_l)) = \sigma(\zeta_l) = P_1,$$
 $$\sigma(P_2) = \sigma(\phi(Q)) = \phi(\sigma(Q)) = \phi(\zeta_lQ) = \phi(\zeta) + \phi(Q) = P_1 + P_2. $$
This completes the split multiplicative reduction case.

The reduction step says the following: If ${L'}/L$ is a finite extension with $l \nmid [{L'}:L]$
and there exists some $\sigma \in I_{\Lbar/{L'}}$ and $P_1,P_2\in E[l](\Lbar)$ such that $\sigma(P_1) = P_1$ and $\sigma(P_2) = P_1 + P_2$ and then $l\nmid \ord_{L'}(j_E)$  the same holds for $L$. 

We give a proof of the reduction step. 
If $l\nmid \ord_{L'}(q)$ then there exists some $\sigma\in I_{\Lbar/L'}$ with $P_1$ and $P_2$ giving a transvection.
Suppose now $l\nmid \ord_L(q)$. 
Since $\ord_{L'}(q) = e({L'}/L) \ord_L(q)$ we get that $l\nmid \ord_{L'}(q)$ unless $l\vert e(L'/L)$.
By hypothesis, there exists some $\sigma\in I_{\Lbar/{L'}}$ and $P_1,P_2 \in E(\Lbar)$ giving a transvection. 
Since $I_{\Lbar/{L'}} \subset I_{\Lbar/L}$ this gives the result. 

General case: suppose $l\nmid \ord_L(q)$. 
Let $L' = L(\sqrt{\gamma},\zeta_l)$ so that $E_{L'}$ has split multiplicative reduction and a root of unity (as in case (1)).\footnote{$\gamma = -c_4/c_6$}
We have that $[L':L] \vert 2(l-1)$ and hence $l\nmid [L':L]$.  
The general result now follows from the reduction step. 
\end{proof}

This now brings us to the main point of this subsection, the ``local canonical splitting'' (which Mochizuki attempts to globalize): If $L/\Q_p$ is a finite extension, and $E/L$ is an elliptic curve with multiplicative reduction then there exists a canonical inclusion $\mu_l(\Lbar) \xrightarrow{\can} E(\Lbar)[l]$ of $G_L$-modules given by the following diagram.
\begin{equation}\label{E:local-multiplicative-subspace}
\xymatrix{
	1 \ar[r] & \mu_l(\Lbar) \ar[d]_{\sim} \ar[r] & E(\Lbar)[l] \ar[r] \ar[d]_{\phi}^{\sim} & \Z/l \ar[r] \ar[d]_{\sim} & 1 \\
	1 \ar[r] & \zeta_l^{\Z} \ar[r] & \zeta^{\Z}_lQ^{\Z}/q^{\Z} \ar[r] & \zeta_l^{\Z} Q^{\Z}/\zeta_l^{\Z} q^{\Z} \ar[r] & 1 		
}.
\end{equation}
We make two important remarks. First, the sequence \eqref{E:local-multiplicative-subspace} does not depend on the choices of $Q$ and $\zeta_l$. 
Second, the generator of the quotient is 
\begin{equation}\label{E:canonical-generator}
	Q \equiv q^{1/l} \ \ \mod \ \langle \zeta_l, q \rangle
\end{equation}
	and this class is uniquely determined.
We record this situation using the following definition.

\begin{definition}
	Let $L$ be a finite extension of $\Q_p$ and $E$ an elliptic curve over $L$ with multiplicative reduction.
	\begin{enumerate}
		\item We call the $G_L$-submodule isomorphic to $\F_l(1):=\mu_l(\Lbar)$ in \eqref{E:local-multiplicative-subspace} the \emph{ (local) multiplicative subspace}.
		\item We call the generator of the quotient of the $G_L$-module $E[l]$ isomorphic to $\Z/l$ (as a $G_L$-module) given in \eqref{E:canonical-generator} the \emph{(local) canonical generator}.
		(In terms of the Tate uniformization, this is just the image of an $l$th root of the Tate parameter.) 
	\end{enumerate}
\end{definition}

Mochizuki wishes to globalize this situation. Consider the situation now an elliptic curve $E$ over a number field $F$ and let $K = F(E[l])$ be the $l$-division field of $F$ associated to $E$.
Fix $M \subset E[l](K)$ isomorphic to $\ZZ/l$ as a subgroup, which we may or may not view it as a $G(K/F)$-module (most of the time it will be a subgroup which is not preserved by $G(K/F)$). 

Let $w \in V(K)_0$ and let $v$ denote the place in $V(F)$ below $w$. 
Completing $K$ at the place $w$ gives the field $K_w$ determines which naturally contains the completion of $F$ at $v$; this gives a diagram of inclusions
 $$ \xymatrix{
 	F \ar[r] \ar[d] & K\ar[d] \\
 	F_v \ar[r] & K_w 
 }.$$  
We will let $G(w/v)= G(K_w/F_v)$ and also recognize it as the stabilizer of $w \in V(K)_v$ under the natural action of $G(K/F)$. 
In this way $G(w/v) \subset G(K/F)$ naturally.
All of this gives $E(K) \subset E(K_w) = E_v(K_w)$ (as groups) which allows us to view $M$ as a subgroup of $E_v(K_w)$.
We will let $M_w$ denote this same group now viewed with a $G(w/v)$-action coming from the action of $G(w/v)$ on $E_v(K_w)$. 
Note that most of the time this is not a $G(w/v)$-submodule --- if $M$ started as a $G(K/F)$-submodule then the restriction of the $G(K/F)$-action to a $G(w/v)$-action would give $M_w$ the structure of a $G(w/v)$-submodule of $E_v[l]$ but these the $G(K/F)$-submodules never exist. 
For future reference, for $P \in E(K)$ we will let $P_w$ denote the same element now viewed as an element in $E_v(K_w)$ (viewed as a $G(w/v)$-module). 
A final piece of notation: for each $w \in V(K)$ we let 
 $$\mu_l(w) \subset E_v(K_w)[l]$$ 
be $G(w/v)$-sub module isomorphic to $\ZZ/l(1)$.
We call this \emph{the local multiplicative subspace of $E$ at $w$}. 

\begin{remark}[WARNING!]
	Note that $M_w$ is not a base change of a group scheme to $K_w$ or an object where we consider automorphisms on it which fix $K_w$. If one thinks in these terms one might confuse $M_w$ with a module that has a trivial Galois action which is not what we want.
\end{remark}

We can now state the globalizability question: is it possible to find some $M \subset E(K)$ such that $M_w = \mu_l(w)$ for each $w \in V(K)$ where $E$ has multiplicative reduction?
The answer is no, and this is stated in Mochizuki's first paper on Hodge-Arakelov theory \cite[Chapter III]{Mochizuki1999}.
Heuristically, there are $(l+1)$-subgroups of $E[l](K)$ isomorphic to $\Z/l$ as groups and if we think of $M \subset E_F[l]$ as random variable selecting some ``random'' $M_w$ uniformly, then the canonical multiplicative subspace should be selected with probability $1/(l+1)$.

Although such a choice of $M$ can't be achieved at all places, the purpose of $(\Vu,\Mu,\epsu)$ is to simulate a ``global multiplicative subspace'' and ``global canonical generator'' by ``rigging'' $\Vu$ so that for each $\vu \in \Vu$, we have $M_{\vu} = \mu_l(\vu)$.
This is achieved in \S\ref{S:existence}. 
We now give some definitions.
\begin{definition}
	Let $w \in V(K)$ be a place of multiplicative reduction for $E_K$ sitting over $v \in V(F)$.
	\begin{enumerate}
		\item Let $M\subset E[l](K)$ be a sub $G(K/F)$-module isomorphic to $\ZZ/l$ as an abelian group.
		We say that $M$ is \emph{a global multiplicative subspace} for $w$ if 
		$$M_w = \mu_l(w)$$ 
		as a $G(w/v)$-module.
		\item Let $\epsu \in \Mu$.
		We say that $\epsu$ is a \emph{(spin) global canonical generator} for $w$ if the image of $\epsu_w$ in the $\Mu_w$ ($\cong \ZZ/l$ by the first part) is $\pm$ the local canonical canonical generator at $w$.
	\end{enumerate}
	If $S \subset V(K)$ we will say that say that $M$ is \emph{a global multiplicative subspace for $S$} if it is for each $w \in S$.
	Similarly, we will say $\epsu$ is a \emph{global canonical generator for $S$} if it is for each $w \in S$. 
\end{definition}

\subsection{Initial Theta Data: ``Fake'' Global Multiplicative Subspaces and Canonical Generators}\label{S:existence}
Let $E$ be an elliptic curve over a number field $F$.
Fix $l$ a prime.  
Let $K=F(E[l])= \Fbar^{\ker(\rho_l)}$ be the field obtained by adjoining the $l$-torsion. 
Consider now the tuple 
 $$(\Vu,\Mu,\epsu).$$
where 
\begin{itemize}
	\item $\Vu \subset V(K)$ is a lift of $V(\Q(j_E))$.
	\item $\Mu \subset \Eu[l](\Kbar)$ is a subgroup. ($\leftrightarrow M \subset E[l](K)$ by Fricke)
	\item $\epsu \in \Mu$ (which we only care about up to $\pm 1$)
\end{itemize}

We can now state the definition of initial theta data. 
\begin{definition}\label{D:initial-theta-data}
	Fix $(F,l,E,S)$ is a tuple of pre theta data.
	Let $K = F(E[l])$. 
	\emph{Initial theta data} is a tuple  
		$$ (F, l, E, S, \Vu, \Mu, \epsu)$$
		where $\Mu$ is the Fricke involute of $M$ a global multiplicative subspace (with respect to $\Vu \cap V(K)_S$) and $\epsu \in \Mu(\Fbar)$ is a global canonical generator (with respect to $\Vu\cap V(K)_S$).
\end{definition}

If $E$ is a member of a tuple of initial theta data we say \emph{$E$ sits in initial theta data}.

We remark that Mochzuki definition environment for initial theta data is \cite[Definition 3.1]{IUT1}. 
It includes a lot of material which is nonstandard for definition environments. 
For example, it contains auxillary instantiation of notation for Galois groups, fields, elliptic curves, and various sets of places. It also includes notation for auxillary covers, fibers of morphism and remarks about equivalences of various auxillary constructions as well as several statements about what can be inferred from the various subdefinitions provided.

\subsection{Existence of Initial Theta Data}
We now prove existence of initial theta data.
We begin by giving initial theta data over a single place of bad reduction.
\begin{lemma}[Simple Initial Theta Data]\label{L:simple-initial-theta-data}
	Let $F$ be a number field. 
	Let $E/F$ be an elliptic curve. 
	Suppose that $j_E \notin \OO_F$.
	Let $K = F(E[l])$ for some rational prime $l$. 
	For all but finitely many choices of $l$, there exists some $P_1, P_2 \in E[l](K)$ and some $\sigma\in G_F$ with such that 
	 $$ \rho_l(\sigma) = \left[ \begin{matrix} 1 & 1 \\ 0 & 1 \end{matrix} \right ]$$
	with respect to this basis. 
\end{lemma}
\begin{proof}
	Compare what follows to \cite[V.6.2]{Silverman2013}.
	Let $w \in V(F)_0$ with $\vert j_E \vert_w >1$. 
	Fix an inclusions $\Fbar \subset \overline{F_w}$ so that we have $G_w \subset G_F$. 
	By the existence of local transvections there exists some $\sigma \in G_w$ and $P_{w,1},P_{w,2} \in E[l](K_w)$ so that $\rho_l(\sigma)$ is unipotent. 
	We need now some $\Mu \subset E[l](K)$ and some $\epsu \in \Mu(K)$ (up to sign) such that after base change they identify with the  $q^{1/l}$.
	The identifications of $G_w\subset G_K$ and $E[l](\Fbar) \subset E(\overline{F_w})$ give the result letting $S = \lbrace w \rbrace$.

\end{proof}

One claim of \cite[Proposition 2.2]{IUT4} is that there exists a global multiplicative subspace $M \subset E[l]$ relative to some $\Vu \subset V(K)$ a section of $V(F_0)$ (in the case when we have more than one bad prime). 
The lemma below verifies this assertion.

\begin{lemma}[Construction of Global Multiplicative Subspaces] \label{L:global-multiplicative}
	For all $M \subset E(\Fbar)[l]$ and all $v \in V(F_0)$ there exists some $\vu \in V(K)_v$ such that $M_{\vu} \cong \mu_l(\vu)$. 
	Performing this operations allows us to construct a collection $\Vu$ for a given subspace $M$ so that $M$ is a global multiplicative subspace with respect to $\Vu$. 
\end{lemma}
\begin{proof}
	Call a pair $(w,M)$ consisting of a place $w \in V(K)_v$ and a subspace $M\subset E[l]$ \emph{good} if $M = \mu(w)$. 
	Note that if a pair $(w,M)$ is good and $\sigma \in G_F$ then $(\sigma(w),\sigma(M))$ is good. 
	This is because if $\sigma(w) = w'$ then $\sigma G(w/v) \sigma^{-1}= G(w'/v)$ and we can write $\tau' = \sigma \tau \sigma^{-1}$ for each $\tau' \in G(w'/v)$. 
	This then shows that the action of $\tau'$ on $\sigma(M)$ performs as advertized since for each $m \in M$
	 $$ \tau' \cdot \sigma(m) = \sigma\tau\sigma^{-1} \sigma m = \sigma \tau(m) = \sigma (\chi_l(\tau) m) = \chi_l(\tau) \sigma(m). $$
	This proves $\sigma(\mu(w)) \cong \mu(\sigma(w))$.
	
	We will now apply the above to construct a good collection of $\Vu$. 
	First let $v \mapsto \widetilde{v}$ be any set theoretic section $V(F) \to V(K)$ of the natural projection $V(K) \to V(F)$.  
	Consider good pairs $(\widetilde{v},\mu(\widetilde{v}))$ where $\mu(\widetilde{v}) \subset E[l](K)$ is the unique subspace which is cyclotomic for $G(\widetilde{v}/v) \subset G(K/F)$. 
	Note that $G_F$ acts transitively on both $V(K)_v$ and the subspaces of $E[l](K) \cong \FF_l^2$ by the hypothesis that $\rho_l(G_F) \supset \SL_2(\FF_l)$. 
	For each $\widetilde{v}$ select some $\sigma_v \in G(K/F)$ such that $\sigma_v(\mu(v))=M$. 
	Define a section $\Vu \subset V(K)$ by 
	 $$ \vu:=\sigma_v(\widetilde{v}). $$
	To define a section $V(F_0)\to V(K)$ take any section $V(F_0) \to V(F)$ and compose it with a section of the type described above.
\end{proof}

\begin{remark}
	\begin{enumerate}
		\item The reader may wish to consult Chapter 2, Lemma 1.3 of Mochizuki's Hodge Arakelov Theory paper \cite{Mochizuki1999} for a criteria for determining when a subspace is multiplicative.
		\item  In \cite[Definition 3.1]{IUT1} Mochizuki uses $\Cu_K$ instead of $\Mu$ and defines $\epsu$ as a cusp of $\Cu_K$ rather that as an element of $\Mu(K)$.
		Here $\Cu_K$ is a hyperbolic orbicurve defined over $K$ (a stack), defined to be the quotient of $\Xu_K$ by the elliptic involution.
		We claim that these are equivalent.
		
		First suppose we are given the data of $\Cu_K$ together with a cusp $\epsu \in \Cu_K^c$. 
		Since we have the maps $\Cu_K \to C_K$ and $X_K \to C_K$ we can form the fiber product $\Xu_K := \Cu_K \times_{C_K} X_K$ which has the isogeny $\Xu_K \to X_K$. 
		The kernel of the compactification of this map is $\Mu$. 
		Conversely, suppose that we have $\Mu$ as (a quotient of $E[l]$) and $\epsu$. 
		This determines the cover $\Eu$ as in the second part of \S\ref{S:fricke-covers}, via the Galois correspondence.	
	\end{enumerate}
\end{remark}

\subsection{Example: Initial Theta Data for E11a1}\label{S:e11a1}
We now produce elements of initial theta data starting from the elliptic curve $E/\Q$ with Cremona label \textsf{11a1}.
This elliptic curve is given by $$E:y^2+y=x^3-x^2-10x-20,$$
and its basic invariants are
\begin{align*}
	c_4(E) &= 496 = 2^4 \cdot 31,\\
	c_6(E) &= 20008 = 2^3\cdot 41 \cdot 61,\\
	j_E &= -1 \cdot  2^{12} \cdot 11^{-5} \cdot 31^3.
\end{align*}
We now go through the ``Simple Initial Theta Data'' computation (Lemma~\ref{L:simple-initial-theta-data}) for this curve explicitly.

\subsubsection{Non-arithmeticity}
As explained in \ref{E:E11a1} this curve is not arithmetic.
In general to obtain arithmeticity one just needs to exclude CM curves and the Sijsling $j$-invariants. 

\subsubsection{Non-empty Collection of Places of Bad Reduction $S=V^{\bad}_{mod}$}
There is only one place of bad reduction. 
That is at $p=11$ and $E$ has split multiplicative reduction at this place.
This means $S = \lbrace 11 \rbrace \subset V(\Q)$.

\subsubsection{Minimal Weierstrass Model at $p=11$}
The local Weierstrass minimal model at $p=11$ of $E$ is 
$$ E^{\min}_{11}:y^2 + y = x^3 - x^2 - 10x - 20.$$
As $E$ is defined over $\Q$ which has class number $1$ we have that this is a global minimal model. 
We record its invariants:
\begin{align*}
	&\mbox{Minimal discriminant valuation}: 5\\
	&\mbox{Conductor exponent}: 1\\
	&\mbox{Kodaira Symbol:} \operatorname{I5}\\
	&\mbox{Tamagawa Number:}  \ \vert \mathcal{E}/\mathcal{E}^0\vert = \vert \ZZ/q \vert =\ord(q) = 5
\end{align*}

\subsubsection{The Tate Uniformization at $p=11$}
The Tate model of this curve is 
$$E_{q_{11}}:y^2 + xy=x^3 +s_4(q_{11})x + s_6(q_{11}),$$
where $q_{11}$ is the Tate parameter. 
We have computed these out to $O(11^{25})$ using \textsf{Sage}\footnote{See \url{http://doc.sagemath.org/html/en/reference/padics/sage/rings/padics/padic_extension_generic.html}
	and \url{http://sporadic.stanford.edu/reference/curves/sage/schemes/elliptic_curves/ell_tate_curve.html}.	
}:
\begin{align*}
	q=q_{11} &= 0. 0, 0, 0, 0, 10, 2, 6, 6, 5, 4, 4, 1, 4, 1, 0, 5, 9, 9, 3, 3, 1, 3, 4,\ldots\\
	s_4 &= 0.0,0,0, 0, 5, 7, 1, 0, 5, 9, 1, 9, 2, 10, 2, 0, 1, 6, 2, 6, 4, 10, 10,\ldots \\
	s_6 &=0. 0, 0, 0, 0, 1, 8, 4, 4, 5, 5, 10, 2, 1, 8, 2, 7, 10, 9, 6, 3, 3, 8, 5,\ldots
\end{align*}
Here the decimal is the beginning of the integral digits and commas separate $11$-adic digits.
We make use of the the uniformization map 
$$ \Qbar_{11}^{\times}/q_{11}^{\ZZ} \xrightarrow{\varphi} E(\Qbar_{11}).$$	
In particular for each $n$ if we let $\zeta_n$ be a primitive $n$th root of unity and $Q_n= q^{1/n}$ be a choice of $n$th root we have
$$ P_{1,n} = \varphi(\zeta_n), P_{2,n} = \varphi(Q_n) \in E[n](\Qbar).$$

\subsubsection{Field $F$} We will let $F = \Q(\sqrt{-1}, E(\Qbar)[30])$. 
To obtain $E(\Qbar)[30]$ ``numerically'' (Turing computable, but not computable in practice) in $\Qbar$ we fix an embedding $\Qbar \subset \Qbar_{11}$.
We than take the coordinates of the Tate uniformization and map them to our model over $\QQ$. 
Observe that $\varphi = f\circ \phi$ where $f:E_{q} \to E$ is the morphism defined over $\Q_{11}$ (since $E$ has split multiplicative reduction over $\Q_{11}$ --- with a global minimal Weierstrass model defined over $\Q$ and $\phi$ the Tate uniformization of $E_q$).
This give $E_q(\Qbar_{11})[30] \cong \zeta^{\Z} Q^{\Z}/q_{11}^{\Z}$ where $Q = q_{11}^{1/30}$ is some choice of $30$th root of $q$ and $\zeta$ is a primitive $l$th root of unity and taking the image of these points under $f$ gives $\QQbar$-points.

\subsubsection{Condition on $l$}
The prime $l=13$ satisfies the divisibility hypotheses of initial theta data. 
We have seen that $\ord(q) = 5$, the only bad place is $p=11$, and that the order of $[F:\Q(j_E)]$ (here $\Q(j_E) = \Q$) can be computed by viewing $F$ as a succession of prime torsion extensions. 
The order of the $r$-torsion extension divides $\vert \GL_2(\F_r) \vert$ for a given prime $r$. 
The relevant orders for $30$-torsion are given below:
\begin{center}
	\begin{tabular}{|c|c|} 
		\hline $r$ & $\vert GL_2(\F_r) \vert = (r^2-1)(r^2-1 -(r-1))$ \\
		\hline 2 & $6 = 2\cdot 3$\\
		\hline 3 & $48 = 2^4 \cdot 3$\\
		\hline 5 & $480 = 2^5\cdot 3\cdot 5$\\ \hline
		
	\end{tabular}
\end{center}

\subsubsection{Surjectivity of $\rho_l$} We require that $\rho_l: G_{F} \to \Aut(E[l])$ contains a copy of  $\SL_2(\F_l)$ relative to some basis. It can be checked, for example, using \textsf{Magma} that the mod $l$ Galois representation for $l>5$ is surjective.
We did this for $l=13$. In fact, the image of $G_{\QQ} \to \GL_2(\ZZhat)$ is the inverse image of $\rho_{550}(G_{\QQ}) \subset \GL_2(\ZZ/550)$ under the natural map $\GL_2(\ZZhat)\to \GL_2(\ZZ/550)$ using the \textsf{OpenImage} algorithm \cite{Zywina2022b}.

\subsubsection{Computation of the Field $K$}
We have $K = F(E_K[l](\Fbar))$ for $l=13$;
There are two approaches to explicit computations. 
One may use the Weierstrass uniformization (at archimedean places) and the Tate uniformization (at both archimedean and non-archimedean places).

\subsubsection{The generator $\underline{\epsilon}$}
We use (the image of) $q^{1/l}$ where $q$ is the Tate parameter at $p=11$ and $l=13$.
These values can be gleaned from the Tate uniformization: $(\QQbar_{11}^{\times}/q^{\ZZ})[13] = (\zeta_{13})^{\ZZ} (q_{11}^{1/13})^{\ZZ}/ q_{11}^{\ZZ}.$
So the torsion field at a local place has the form $K_{\vu}=\QQ_{11}(\zeta_6,\zeta_{13}, q_{11}^{1/13}, q_{11}^{1/6})$.

\subsubsection{ A Remark}
\begin{remark}
	One could hope to generate simple example of initial theta data where $K$ is small by looking for some $E$ over $\Q$ with many $\Q$-rational torsion points. 
	This approach fails because of Mazur's theorem.
	Mazur's Theorem says that if $E$ is an elliptic curve defined over $\Q$ the possible torsion subgroups of $E(\Q)$ are $C_1,C_2,\ldots,C_{10},C_{12}, C_2 \oplus C_4,C_2\oplus C_4,C_2\oplus C_6,C_2\oplus C_8$.
	Here $C_n$ denotes the cyclic group of order $n$ \cite[Theorem 8]{Mazur1977}.
	
	We would also like to point out that Szpiro's inequality implies a uniform computable bound on the number of torsion points in a field $F$ \cite[Exercise 5.16]{Silverman2013}.
\end{remark}

\subsection{Algorithm for Initial Theta Data}\label{S:algorithm}
Algorithm~\ref{A:initial-theta-data} below gives a procedure for computing initial theta data given an elliptic curve $E_0$ defined over $\QQ$ in \textsf{Sage} pseudocode. 
This algorithm will either return \textsf{False} if $E_0$ is not admissible (for example it has a Sijsling $j$-invariant) or it will return $F_2,\ell, \Vu, \Mu, \epsu$ as in the initial theta data. 

Algorithm~\ref{A:initial-theta-data} depends on the following functions 
\begin{itemize}
	\item \textsf{get\_l()}, which is described in Algorithm~\ref{A:get-ell} and returns a small $\ell$ which is valid for initial theta data given the input is an admissible elliptic curve
	\item \textsf{get\_global\_multiplicative\_subspace()}, a procedure described in Lemma~\ref{L:global-multiplicative} for constructing global multiplicative subspaces.
	\item \textsf{is\_sisjling()}, which tests for ``bad'' $j$-invariants and is described in Theorem~\ref{T:sijsling}.
	\item Functions from the \textsf{Sage} library which are documented at \url{https://www.sagemath.org}.
\end{itemize}

While \textsf{get\_l()} and Algorithm~\ref{A:initial-theta-data} are Turing computable, they are not practical as they rely on computations of large splitting fields. 
Line 12 of Algorithm~\ref{A:initial-theta-data} asks us to compute the $30$-torsion field of an elliptic curve and similarly the function \textsf{get\_global\_multiplicative\_subspace}() relies on this splitting field. 

What about determining $\ell$? Checking everything except for surjectivity of the Galois representation over $\QQ(E[30],i)$ is practical and one might hope that one could test surjectivity of the the mod $\ell$ representation over $\QQ$ or $\QQ(i)$ and that it would imply surjectivity of the mod $\ell$ representation over $F=\QQ(i,E[30])$ (surjectivity is overkill but sufficient for initial theta data). Unfortunately, the interaction between $E[a]$ and $E[b]$ for $a,b\in \ZZ_{\geq 1}$ can be complicated due to ``entanglements''; the image of the mod $\ell$ representation over $\QQ(i,E[30])$ may be smaller than image of the the mod $\ell$ representation over $\QQ(i)$; in fact there is a large program to classify such entanglements \cite{Daniels2023b}.

Fortunately, there do exist modern computational developments surrounding the images of Galois representations of elliptic curves which allows us to resolve our issues with entanglements.
We first explain abstractly why such an $\ell$ exists and then give an algorithm suggested to us by David Zureick-Brown. 

Let $E$ be a semi-stable elliptic curve over $\QQ$.
For a field $L$ and a some $n \in \ZZ_{\geq 1}$ we will let $\rho_{n,L}$ denote the representation $G(\overline{Q}/L)\to \Aut(E[n]) \cong \GL_2(\ZZ/n)$. 
Let $F = Q(i,E[30])$.
We seek to find a small prime $\ell$ such that $\rho_{\ell,F}$ is surjective.

By Serre's Open Image Theorem \cite[pg 260, item (2)]{Serre1972} the map $\rho_{\QQ}: G(\overline{\QQ}/\QQ) \to \GL_2(\widehat{Z})$ has an open image and hence is finite index.
This implies that all but finitely many of its images $\rho_{\QQ,\ell^{\infty}}$ into $\GL_2(\ZZ_{\ell})$ must be all of $\GL_2(\ZZ_{\ell})$. Hence that for all but finitely many $\ell$, the mod $\ell$ image is all of $\GL_2(\ZZ/\ell)$.
\footnote{See also \cite[Corollary 2.4]{Zhou2025}}

\begin{theorem}[\cite{Rouse2022},\cite{Zywina2022}]
	There exists an algorithm \textsf{OpenImage} which can compute the image of $G_{\QQ}\to \GL_2(\ZZhat)$ explicitly for any semi-stable elliptic curve.
\end{theorem}
The algorithm accompanying \cite{Zywina2022} is implemented in \textsf{Magma} can be found on github \cite{Zywina2022b}.\footnote{There are several forks of this repository. See the David Lowry-Duda, Andrew Sutherland, or David Roe's github pages.}

Here is how it works: for each elliptic curve $E$, \textsf{OpenImage} describes the image of $G_{\QQ}$ in $\GL_2(\ZZhat)$  as the inverse image of some explicit finitely generated subgroup of $\GL_2(\ZZ/M\ZZ)$ (where $M\in \ZZ_{\geq 1}$ depends on $E$) under the surjection $\GL_2(\ZZhat)\to\GL_2(\ZZ/M)$. That is it.
Note that in particular given prime $\ell$ if $\ell \nmid M$ then the mod $\ell$ Galois representation is surjective. Also observe that this algorithm computes all the Galois representations simultaneously so there is no need to do more computations as you search for larger primes.

The following is a description of how to use \textsf{OpenImage} to find an admissible prime $\ell$ for initial theta data.

Let $E./\QQ$ be an elliptic curve and $L$ be a number field and $\ell$ a prime.
Suppose two conditions: 1)  $Q(E[\ell]) \cap L = \QQ$, 2) $\rho_{\ell,\QQ}$ surjective. Then $\rho_{\ell,L}$ is surjective. 
This is because the condition that an automorphism fixes $L$ has no effect on the torsion points due to linear disjointness; the image of $G_L$ is the same as the image of $G_{\QQ}$.
Second, two Galois extensions $L_1$ and $L_2$ of a base field $L_0$ are linearly disjoint if and only if $G(L_1L_2/L_0) = G(L_1/L_0)\times G(L_2/L_0)$. This is because $[L_1L_2:L_0] = [L_1:L_0][L_2:L_0]$.
This can be tested explicitly in the case that $L_1 = \QQ(E[a])$ and $L_2=\QQ(E[b])$ for $a$ and $b$ coprime positive integers using the \textsf{OpenImage} as follows: 
Let $G_n$ denote the image of $\rho_{n,\QQ} $ for an integer $n \in \ZZ_{\geq 2}$. 
Note that $G_n$ is the Galois group of $\QQ(E[n])$ over $\QQ$.
The linear disjointness criterion of $\QQ(E[a])$ and $\QQ(E[b])$ for $a,b\in \ZZ_{\geq 2}$coprime then reads $\vert G_{ab} \vert = \vert G_a \vert \cdot \vert G_b \vert$.

Finally, to find a prime $\ell$ such that $\rho_{\ell,F}$ is surjective for $F=\QQ(\sqrt{i},E[30])$ we note that $\QQ(i) \subset \QQ(E[4])$ so that $F \subset \QQ(E[60])$ and then proceed to find the smallest $\ell$ such that the following two conditions are satisfied. 1) $\rho_{\QQ,\ell}$ is surjective. 2) $\QQ(E[\ell])$ is linear disjoint from $\QQ(E[60])$ (which is read off from the Galois groups as $\vert G_{60 \ell}\vert = \vert G_{\ell}\vert \cdot \vert G_{60} \vert$).

\begin{example}
Starting with the Elliptic Curve defined by $y^2 + y = x^3 + x^2 - 12x - 21$ over $\QQ$ of conductor $N=67$ we first see that the conductor is not divisible by $\ell=7$.
One then checks the Tate parameter:
 $$q=52\cdot 67 + 3\cdot 67^2 + 15\cdot 67^3 + 3\cdot 67^4 + 3\cdot 67^5 + 17\cdot 67^6 + 38\cdot 67^7 + 18\cdot 67^8 + O(67^9) ,$$ 
and sees it does not have a valuation not divisible by $\ell=7$. One then finds that its  $7$-adic galois representation surjective so $\ell=7$ works for this elliptic curve. 
\end{example}

In the implementation, $5525$ elliptic curves in the Cremona tables of prime conductor there exists some $\ell \in \lbrace 7,11,13\rbrace$ such that $\rho_{\ell,F}$ is surjective where $F=\QQ(i,E[30])$.

\begin{remark}
	Subsequent to Mochizuki’s 2021 work, the paper of Mochizuki, Fesenko, Hoshi, Minamide, and Porowski \cite{Mochizuki2022} has modified versions of initial theta data and in particular can work with an arbitrary set of bad places.
\end{remark}

\begin{remark}
	Explicit initial theta data from \cite{Mochizuki2022} appears in the preprint \cite{Zhou2025}.
\end{remark}

\newpage

\begin{algorithm}[H]\label{A:get-ell}
	\DontPrintSemicolon
	\SetKwProg{Fn}{def}{\string:}{}
	\SetKwData{jzero}{j0}
	\SetKwData{Fone}{F1}
	\SetKwData{Ftwo}{F2}
	\SetKwData{Ezero}{E0}
	\SetKwData{Eq}{Eq}
	\SetKwData{Eone}{E1}
	\SetKwData{l}{l}
	\SetKwData{Nzero}{N0}
	\SetKwData{Primes}{primes}
	\SetKwData{FoundEll}{found\_ell}
	\SetKwData{Deltazero}{Delta0}
	\SetKwData{badPlaces}{bad\_places}
	\SetKwData{True}{True}
	\SetKwData{False}{False}
	\SetKwData{Break}{Break}
	\SetKwData{cc}{c}
	\SetKwFunction{J}{j\_invariant}
	\SetKwFunction{getThetaData}{get\_theta\_data}
	\SetKwFunction{BaseChange}{base\_change}
	\SetKwFunction{NumberField}{NumberField}
	\SetKwFunction{FixBreak}{FixBreak}
	\SetKwFunction{IsSijsling}{is\_sijsling}
	\SetKwFunction{IsSemistable}{is\_semistable}
	\SetKwFunction{NextPrime}{next\_prime}
	\SetKwFunction{Conductor}{conductor}
	\SetKwFunction{PrimeDivisors}{prime\_divisors}
	\SetKwFunction{MinimalDiscriminant}{minimal\_discriminant}
	\SetKwFunction{TateCurve}{tate\_curve}
	\SetKwFunction{Parameter}{parameter}
	\SetKwFunction{Valuation}{valuation}  
	\SetKwFunction{GetEll}{get\_l}  
	\SetKwFunction{GaloisRepresentation}{galois\_representation}  
	\SetKwFunction{IsSurjective}{is\_surjective}  
	\SetKwFunction{TorsionField}{torsion\_field}  
	\SetKwFunction{GetGlobalMultiplicativeSubspace}{get\_global\_multiplicative\_subspace}

	\Fn{\GetEll{$E_0$,$N_0$}}{
		
		\KwInput{$E_0$ an elliptic curve Sage object over $\QQ$, $N_0$ the conductor of an elliptic curve}
		\KwOutput{$\ell$ a prime that can appear in initial theta data} 
		
		\badPlaces \colonequals \PrimeDivisors{$N_0$} \\
		$\ell$ \colonequals 5\\
		\FoundEll \colonequals \False\\
		
		\While{\FoundEll = \False}
		{
			$\ell$ \colonequals \NextPrime{$\ell$}\\
			\FoundEll \colonequals \True
			
			
			
			\If{ $\ell \vert N_0$ }{
				\FoundEll \colonequals \False
			}
			
			
			\If{ \FoundEll }{
				\For{ $p$ in \badPlaces}{
					$E_q$=$E_0$.\TateCurve{$p$}\\
					\If{ $\ell \mid$ $E_q$.\Parameter{}.\Valuation{} }{
						\FoundEll \colonequals \False \\
						\Break
					}
				}	
				
			}
			
			\If{\FoundEll}{
				$F_1$.$\langle b\rangle$ \colonequals \NumberField{$x^2+1$}\\
				$E_1$ \colonequals $E_0$.base\_change($F_1$)\\
				$F_2$.$\langle c \rangle$ \colonequals $E_1$.\TorsionField{$30$}\\
				$E_2$.base\_change($F_2$)\\
				$\rho$ \colonequals $E_2$.\GaloisRepresentation{}\\
				\If{ not $\rho$.\IsSurjective{$\ell$} }{
					\FoundEll \colonequals \False
				}
				
			}
			
		}
		\Return{ $\ell$}
		
	}
	\caption{An algorithm for computing admissible $\ell$}
\end{algorithm}

\newpage
\begin{algorithm}[H]\label{A:initial-theta-data}
	\DontPrintSemicolon
	\SetKwProg{Fn}{def}{\string:}{}
	\SetKwData{jzero}{j0}
	\SetKwData{Fone}{F1}
	\SetKwData{Ftwo}{F2}
	\SetKwData{Ezero}{E0}
	\SetKwData{Eq}{Eq}
	\SetKwData{Eone}{E1}
	\SetKwData{l}{l}
	\SetKwData{Nzero}{N0}
	\SetKwData{Primes}{primes}
	\SetKwData{FoundEll}{found\_ell}
	\SetKwData{Deltazero}{Delta0}
	\SetKwData{badPlaces}{bad\_places}
	\SetKwData{True}{True}
	\SetKwData{False}{False}
	\SetKwData{Break}{Break}
	\SetKwData{cc}{c}
	\SetKwFunction{J}{j\_invariant}
	\SetKwFunction{getThetaData}{get\_theta\_data}
	\SetKwFunction{BaseChange}{base\_change}
	\SetKwFunction{NumberField}{NumberField}
	\SetKwFunction{FixBreak}{FixBreak}
	\SetKwFunction{IsSijsling}{is\_sijsling}
	\SetKwFunction{IsSemistable}{is\_semistable}
	\SetKwFunction{NextPrime}{next\_prime}
	\SetKwFunction{Conductor}{conductor}
	\SetKwFunction{PrimeDivisors}{prime\_divisors}
	\SetKwFunction{MinimalDiscriminant}{minimal\_discriminant}
	\SetKwFunction{TateCurve}{tate\_curve}
	\SetKwFunction{Parameter}{parameter}
	\SetKwFunction{Valuation}{valuation}  
	\SetKwFunction{GaloisRepresentation}{galois\_representation}  
	\SetKwFunction{IsSurjective}{is\_surjective}  
	\SetKwFunction{TorsionField}{torsion\_field}  
	\SetKwFunction{GetGlobalMultiplicativeSubspace}{get\_global\_multiplicative\_subspace}
	\SetKwFunction{GetEll}{get\_l}

	\Fn{\getThetaData{$E_0$}}{
		
		\KwInput{$E_0$ an elliptic curve Sage object over $\QQ$.}
		\KwOutput{$\ell$ a prime, $F_2$ a field object, $\epsu$ a torsion point,
			$\Mu$ a subgroup of $\ell$-torsion of $E_0$,
			$\Vu$ a collection of places; \False if $E_0$ is not admissible. } 
		$j_0$ = $E_0$.\J{}\\
		$F_1$.$\langle a \rangle$ = \NumberField($x^2+1$)\\
		$E_1$ = $E_0$.\BaseChange{$F_1$}

		\If{ \IsSijsling{$j_0$}}{
			\Return{\False}
		}
		
		\If{ not $E_1$.\IsSemistable{}}{
			\Return{\False}
		}
		
		$N_0$ \colonequals $E_0$.\Conductor{}\\
		$\ell$ \colonequals \GetEll{$E_0$,$N_0$}
		
		\If{\FoundEll}{
			$F_2$.$\langle c \rangle$ \colonequals $E_1$.\TorsionField{$30$} \\
			$\Vu,\Mu,\epsu$ \colonequals \GetGlobalMultiplicativeSubspace{$E_0$,$\ell$} 
			
			\Return{$F_2$,$\ell$,$\Vu$,$\Mu$,$\epsu$}

		}
		
	}
	\caption{An algorithm for computing initial theta data}
\end{algorithm}

\bibliographystyle{amsalpha}
\bibliography{IUTv3.bib}

\appendix

\section{Acknowledgements}
This article is very much indebted to  \cite{Fesenko2015,Hoshi2018, Kedlaya2015,Hoshi2015,Stix2015,Mok2015,Mochizuki2017,Yamashita2017, Hoshi2017,Tan2018,SS}. 
We would like to thank John Cremona, John Voight, Will Chen, and Jereon Sijsling for productive conversations about Fuchsian groups. 
We would also like to thank Jeff Lagarias for many suggestions that have improved this manuscript.
We would like to thank David Zureick-Brown for suggesting the application of \textsf{OpenImage} for computing admissible $\ell$.

The authors also benefitted from the existence of the following workshops: the 2015 Oxford workshop funded by the Clay Mathematics Institute and the EPSRC programme grant \emph{Symmetries and Correspondences}; the 2017 Kyoto \emph{IUT Summit} workshop funded by RIMS and EPSRC; the Vermont workshop in 2017 funded by the NSF DMS-1519977 and \emph{Symmetries and Correspondences} entitled \emph{Kummer Classes and Anabelian Geometry}; the 2018 Vermont Workshop on \emph{Witt Vectors, Deformations and Absolute Geometry} funded by NSF DMS-1801012. 

The first author was supported by the European Research Council under the European Unions Seventh Framework
Programme (FP7/2007-2013) / ERC Grant agreement no. 291111/ MODAG; both authors were supported by NSF DMS-2401570.

The research discussed in the present paper profited from the generous support of the International Joint Usage/Research Center (iJU/RC) located at Kyoto Universities Research Institute for Mathematical Sciences (RIMS) as well as the Preparatory Center for Research in Next-Generation Geometry located at RIMS.
\end{document}